\documentclass[10pt,twoside,a4paper]{amsart}

\usepackage[english]{babel}
\usepackage[utf8]{inputenc}

\usepackage{amsfonts,amsthm,amssymb,amsmath,amsthm,latexsym,wasysym,stmaryrd,mathrsfs,dsfont,txfonts,xcolor}
\usepackage{accents,multirow,rotating,subfigure,graphicx,bigstrut,lscape,multicol,fancyhdr,enumerate,accents,mathtools,exscale,afterpage,scrextend}

\usepackage[normalem]{ulem}

\usepackage{pdftricks}
\usepackage{color}
\usepackage[pdftex,all]{xy}

\usepackage{hyperref}
\usepackage{appendix}
\usepackage[justification=centering]{caption}

\graphicspath{{Figures/}}

\numberwithin{equation}{section}

\theoremstyle{theorem}
\newtheorem {theo}{Theorem}[section]
\newtheorem*{theo*}{Main Theorem}
\newtheorem*{thm*}{Theorem}
\newtheorem {lem}[theo]{Lemma}
\newtheorem*{lem*}{Lemma}
\newtheorem {prop}[theo]{Proposition}
\newtheorem*{prop*}{Proposition}

\newtheorem*{cor*}{Corollary}

\newtheorem*{conjecture*}{Conjecture}
\theoremstyle{definition}
\newtheorem {defi}[theo]{Definition}
\newtheorem*{defi*}{Definition}
\newtheorem {nota}[theo]{Notation}
\newtheorem*{nota*}{Notation}
\theoremstyle{remark}
\newtheorem {rem}[theo]{Remark}
\newtheorem*{rem*}{Remark}
\newtheorem {rems}[theo]{Remarks}

\newtheorem*{warning*}{Warning}

\newtheorem*{warnings*}{Warnings}

\newtheorem*{convention*}{Convention}
\newtheorem {exemple}[theo]{Example}
\newtheorem*{exemple*}{Example}

\newtheorem*{exemples*}{Examples}

\newtheorem*{question*}{Question}

\newtheorem*{questions*}{Questions}

\newtheorem*{fact*}{Fact}

\newtheorem*{acknowledgments*}{Acknowledgments}

\def\1{\mathbf{1}}

\def\N{{\mathds N}}

\def\Z{{\mathds Z}}

\def\2Z{{\fract{\Z}{2\Z}}}
\def\ov{\overline}
\def\e{\varepsilon}

\newcommand{\fract}[2]{\hbox{\leavevmode 
\kern.1em \raise .25ex \hbox{\the\scriptfont0 $#1$}\kern-.1em }\big/
  {\hbox{\kern-.15em \lower .5ex \hbox{\the\scriptfont0 $#2$}} }}
\newcommand{\subfract}[2]{\hbox{\leavevmode
  \kern0em \raise .25ex \hbox{\the\scriptfont0 \tiny $#1$}\kern-.1em }/
  {\hbox{\kern-.15em \lower .5ex \hbox{\the\scriptfont0 \tiny $#2$}} }}

\newcommand{\dessin}[2]{
  \vcenter{\hbox{\includegraphics[height=#1]{#2.pdf}}}}

\renewcommand{\quote}[1]{`#1'}

\def\nR{\textnormal R}
\def\SL{\textnormal SL}
\def\wSL{w\SL}

\DeclareMathOperator{\Aut}{Aut}

\definecolor{pink}{rgb}{0.858, 0.188, 0.478}
\definecolor{orange}{rgb}{1, 0.647, 0}

\begin{document} 

\title{$k$-reduced groups and Milnor invariants} 
\author[B. Audoux]{Benjamin Audoux}
         \address{Aix Marseille Univ, CNRS, Centrale Marseille, I2M, Marseille, France}
         \email{benjamin.audoux@univ-amu.fr}
\author[J.B. Meilhan]{Jean-Baptiste Meilhan} 
\address{Univ. Grenoble Alpes, CNRS, Institut Fourier, F-38000 Grenoble, France}
	 \email{jean-baptiste.meilhan@univ-grenoble-alpes.fr}
\author[A. Yasuhara]{Akira Yasuhara} 
\address{Faculty of Commerce, Waseda University, 1-6-1 Nishi-Waseda,
  Shinjuku-ku, Tokyo 169-8050, Japan}
	 \email{yasuhara@waseda.jp}
%
%
\begin{abstract} 
We characterize, in an algebraic and in a diagrammatic way,
 Milnor string link invariants indexed by sequences
where any index appears at most $k$ times, for any fixed $k\ge 1$. The algebraic characterization is given in terms of an Artin-like action on the so-called $k$--reduced free groups; the diagrammatic characterization uses the langage of welded knot theory. The link case is also addressed. 
\end{abstract} 

\maketitle

\section*{Introduction}

In his seminal works \cite{Milnor,Milnor2}, J.~Milnor introduced a family of concordance invariants for  $n$-component links,  which can be seen as a wide generalization of the linking number. 
Indexed by sequences $I$ of possibly repeating indices in $\{1,\ldots,n\}$, these \emph{Milnor
  invariants} $\mu(I)$ are
integers extracted from longitudes within the fundamental group of the
link complement, well-defined only modulo a subtle indeterminacy. The
geometric meaning of this indeterminacy was clarified by the work of
N.~Habegger and X.S.~Lin \cite{HL}, who showed that Milnor invariants
are actually well-defined integer-valued invariants of \emph{string links}, which are pure tangles without closed components. 

Recall that two (string) links are \emph{link-homotopic} if they are related by a sequence of isotopies and crossing changes involving two strands of a same component. 
Milnor proved in \cite{Milnor2} that, if a sequence $I$ is without repetition, then Milnor invariant $\ov{\mu}(I)$ is in fact a link-homotopy invariant. He further showed how these non-repeated invariants can be extracted from the \emph{reduced} fundamental group of the link complement, which is is the  \quote{maximal}  quotient where each meridian commutes with any of its conjugates. 
Using an Artin-like action of string links on the reduced free group, Habegger and Lin actually showed that two string links have same Milnor invariants indexed by sequences without repetition, if and only if they are link-homotopic \cite{HL}.

The purpose of this paper is to give a higher-order version of this classification result of Habegger and Lin.   Namely, we  characterize the information contained by Milnor invariants $\mu(I)$ of string links with $r(I)\le k$, where $r(I)$ denotes the maximum number of time that any index appears in $I$.  (In particular, Milnor invariants $\mu(I)$  with $r(I)=1$ are precisely Milnor link-homotopy invariants.)  
\\
Our main result is the following; explanations of terminologies and notation shall follow. 
\begin{theo*}
 Let $L$ and $L'$ be two $n$-component string links. The following are equivalent. 
 \begin{enumerate}
  \item $\mu_L(I)=\mu_{L'}(I)$ for any sequence $I$ with $r(I)\le k$. 
  \item $L$ and $L'$ induce the  same $k$--reduced free action: $\varphi_L^{(k)}=\varphi_{L'}^{(k)}\in \Aut_c(R_{k}F_n)$. 
  \item $L$ and $L'$ are self $w_{k}$-concordant. 
\end{enumerate}
\end{theo*}

Let us first explain assertion (2). 
Let $G(L)$ be the fundamental group of the complement of $L$. This
group is normally generated by $n$ meridians,  
and for each $i$, we denote by $N_i$ the normal subgroup of $G(L)$ generated
by the $i$th meridian. 
The \emph{$k$--reduced quotient} of $G(L)$ is defined by 
\[
  \nR_kG(L):= \fract{G(L)}{\Gamma_{k+1} N_1\cdots \Gamma_{k+1} N_n},
  \]
where $\Gamma_q N$ denotes the $q$th term of the lower central series of a group $N$.
This generalizes Milnor's above-mentioned notion of reduced group
\cite{Milnor}, which coincides with $R_1G(L)$. 
In particular, if $F_n$ is the free group on $n$ generators, then $\nR_kF_n$ is called the \emph{$k$--reduced free group}.
Now, given an $n$-component string link $L$, we show in Section \ref{sec:kred} that, for each $k\ge 2$, there is an associated \emph{$k$--reduced free action} 
\[
  \varphi^{(k)}_L\in \Aut_c(R_{k}F_n),
  \]
where $\Aut_c(R_{k}F_n)$ denotes the set of automorphisms of $R_{k}F_n$ that act by conjugation on each generator. 
This can be seen as a generalization of Habegger-Lin's representation for the group of link-homotopy classes of string links \cite[Thm.~1.7]{HL} in the case $k=1$. 
We stress, however, that our proof is very different in nature from \cite{HL}.

Let us now clarify assertion (3). 
We use the langage of \emph{welded knot theory}, which is a diagrammatic generalization of knot theory. 
We stress, firstly, that the set of classical (string) links injects into the larger set $\wSL(n)$ of welded (string) links and, secondly, that Milnor invariants extend naturally to welded objects, so that they coincide with the classical invariants on classical objects. \\
In \cite{arrow} a diagrammatic calculus for welded knotted objects was developped, called \emph{arrow calculus}, which can be seen as a generalization of the theory of Gauss diagrams.  
A $w$-tree for a welded diagram $D$, is an oriented, unitrivalent tree, whose univalent vertices lie disjointly on $D$. 
Given such a $w$-tree $T$ for $D$, there is a \quote{surgery} 
procedure that yields a new welded diagram $D_T$, which is roughly obtained by inserting an \quote{iterated commutator of crossings} 
in a neighborhood of the head of $T$. 
Define the degree of a $w$-tree to be half the number of its vertices. 
The \emph{self $w_k$-equivalence} is the equivalence relation on welded string links generated by surgeries on degree $\ge k$ $w$-trees whose univalent vertices all lie on the \emph{same} string link component. For $k=1$, this notion turns out to coincide with the usual link-homotopy relation for string links \cite{ABMW}. 
On the other hand, there is a combinatorial equivalence relation of
\emph{welded concordance} for welded (string) links, which is
generated by birth, death and saddle moves \cite{BC,Gaudreau} and
which naturally encompasses the topological concordance for classical (string) links. 
The \emph{self $w_k$-concordance} is the equivalence relation obtained by combining the above two relations, and our main result states that this characterizes combinatorially the information contained in Milnor invariants $\mu(I)$ with $r(I)\le k$. 
 \medskip 

In fact, the $k$-reduced free action involved in assertion (2) is more generally defined for $wSL(n)$, and our main theorem follows from a more general characterization for welded string links, see Theorems \ref{thm:12} and \ref{th:23}. 
 We also show that the map sending a welded string link $L$ to its associated $k$-reduced free action $\varphi^{(k)}$, descends to a group isomorphism 
 \[
   \fract{\wSL(n)}{\textrm{self
       $w_k$-concordance}}\stackrel{\simeq}{\longrightarrow} \Aut_c(R_{k}F_n),
   \]
 for any $k\ge 1$; see Proposition \ref{prop:iso}. The case $k=1$ was proved in \cite{ABMW}, and is a generalization of \cite[Thm.~1.7]{HL}. 
This isomorphism suggests that welded theory provides a sensible diagrammatic counterpart of the algebraic constructions underlying Milnor invariants. 
Our main theorem also refines a recent result of B.~Colombari \cite{Colombari}, who gave a diagrammatic characterization of string links having same Milnor invariants of length $\le q$. 
 We note that the geometric properties of Milnor link invariants $\ov{\mu}(I)$ with $r(I)\le k$ was previously investigated in \cite{FY, yasuharaTAMS,yasuharaAGT}, using clasper theory. 
\medskip 

We address the (welded) link case in the final section of this paper. 
As developped there, building on the proof  of our main theorem for string links, and using straightforward adaptations of the above mentioned work of Colombari \cite{Colombari}, we obtain in particular the following for classical links.

\begin{thm*}
 A link has vanishing Milnor invariants $\ov{\mu}(I)$ with $r(I)\le k$, if and only if it is self $w_k$-equivalent to the trivial link. 
\end{thm*}
This follows from a more general result (Theorem \ref{thm:links}) which characterizes the so-called $k$--reduced peripheral system of \emph{welded} links, that is the $k$-reduced link group endowed with peripheral elements, see Section \ref{sec:links}.

\begin{acknowledgments*}
The authors thank Jacques Darn\'e for bringing to their knowledge the
result  \cite[Lem.~2.37]{darne}. 
The first author is partially supported by the project SyTriQ (ANR-20-CE40-0004) of the ANR, and thanks the IRL PIMS-Europe for its hospitality during the period in which part of the work on this paper was done. 
The second author is partially supported by the project AlMaRe (ANR-19-CE40-0001-01) of the ANR. 
The third author is supported by the
JSPS KAKENHI grant 21K03237, and by the Waseda University grant
for Special Research Projects 2021C-120. 
\end{acknowledgments*}

\section{Basic algebraic preliminaries}\label{sec:1}

This section reviews algebraic tools that will be used in this paper, together with basic and well-known results. 
Several notation used throughout the paper will also be set in this section.

Let $n$ be a positive integer. We denote by $F_n$ the free group on $n$ generators $\alpha_1,\cdots,\alpha_n$. 
For each $i\in \{1,\cdots,n\}$, denote by $N_i$ the normal subgroup of $F_n$ generated by $\alpha_i$.

\subsection{Commutators and the lower central series}\label{sec:LCS}

We use the following convention for commutators and conjugates ($x,y\in F_n$): 
\[
  [x,y]:=x\overline{y}\,\overline{x}y\quad \textrm{ and }\quad  x^y:=
  \overline{y}xy,
  \]
where we write $\overline{a}$ for the inverse of an element $a$.
This somewhat nonstandard convention for commutators will be justified by the diagrammatic counterpart of the theory, reviewed in Section \ref{sec:arrow}. 
We note that with our convention, for elements $a,b,c$ of $F_n$, we
have the following basic properties: \\[-1.1cm]
\begin{multicols}{3}

\begin{equation}\tag{C$0$}\label{eq:cominv}
  \ov{[a,b]}=[\ov{b},\ov{a}];
\end{equation}

\begin{equation}\tag{C$1$}\label{eq:com0}
  [a,b]=\ov{b}^{\ov{a}}b=a\ov{a}^{b};
\end{equation}

\begin{equation}\tag{C$2$}\label{eq:com1}
  [a,b]^c=[a^c,b^c]; 
\end{equation}

\end{multicols}
Moreover, recall the following commutator identies (compare with \cite[Thm.~5.1]{MKS}).  
\begin{equation}\tag{C$3$}\label{eq:com2}
 [a,bc] = [a,c][a,b]^c \,\,\, \textrm{ and } \,\,\, [ab,c]=[b,c]^{\ov{a}}[a,c]; 
\end{equation}
\begin{equation}\tag{C$4$}\label{eq:com3}
 \big[[a,b],c\big] = 
 \big[\ov{a},[\ov{c},\ov{b}]\big]^{b\ov{a}}  \big[\ov{b},[\ov{a},\ov{c}]\big]^{c\ov{a}}.
\end{equation}

The lower central series of $F_n$ is the family of nested subgroups $\{\Gamma_k F_n\}_k$ defined inductively by $\Gamma_1 F_n=F_n$ and $\Gamma_{k+1} F_n = [F_n,\Gamma_k F_n]$.  
The following, less standard, notion will also be useful. 
\begin{defi}\label{def:linear}
A length $k$ \emph{linear commutator} is a an element of $\Gamma_k F_n$ of the form 
\[\left[x_1,\left[x_2,\left[x_3,\cdots
        [x_{k-1},x_k]\cdots\rule{0mm}{.3cm}\right]\rule{0mm}{.4cm}\right]\rule{0mm}{.5cm}\right]
\]
for some elements $x_i\in F_n$. 
\end{defi}

We shall need the following basic results. 
\begin{lem}\label{lem:2.2}
  Let $C\in F_n$ be a length $l$ commutator, with $k$ entries in $N_i$ for some $i$ ($k\leq l$). Then $C$ is a product of length $\geq l$ commutators where each entry is an element of 
$\{\alpha_j,\ov{\alpha}_j\}_j$, and with at least $k$ entries that are either $\alpha_i$ or its inverse. 
\end{lem}
\begin{proof}
Combining  (\ref{eq:com2}) with (\ref{eq:cominv}) and (\ref{eq:com0}), gives the identities 
\begin{equation}\tag{C$5$}\label{eq:combonus}
 [a,bc] = [a,c][a,b]\big[[\ov{b},\ov{a}],c\big] \,\,\, \textrm{ and } \,\,\, [ab,c]=\big[a,[\ov{c},\ov{b}]\big][b,c][a,c]. 
\end{equation}
By assumption, $k$ entries of $C$ are products of conjugates of $\alpha_i$ or its inverse. 
Using (\ref{eq:combonus}) on these $k$ entries, $C$ can be written as
a product of length $\ge l$ commutators, each having $\ge k$ entries that are a single conjugate of $\alpha_i$ or its inverse. 
Since $a^b=a[\overline{a},b]$ by (\ref{eq:com0}), we have by using
(\ref{eq:combonus}) that $C$ is written as a product of length $\geq l$ commutators, each having at least $k$ entries that are either $\alpha_i$ or its inverse. 
Now consider one such length $\geq l$ commutator $C'$. One can apply (\ref{eq:combonus}) iteratively on 
all entries of $C'$, until it is written as a product of iterated
commutators with entries in $\big\{\alpha_j;\ov{\alpha}_j\big\}_j$ and clearly, each of these commutators necessarily contains at least $k$ entries that are either $\alpha_i$ or its inverse, since $C'$ does.
\end{proof}
\begin{lem}\label{lem:2.1}
Let $C\in F_n$ be a length $l$ commutator, with $k$ entries in $N_i$ 
for some $i$ ($k\leq l$). 
Then $C\in \Gamma_kN_i$. More precisely, $C$ is a product of length $k$ linear commutators whose entries are all conjugates of $\alpha_i$  or its inverse. 
\end{lem}
\begin{proof}
  By assumption, $C$ is a length $l$ commutator with at least $k$ entries that are products of conjugates of $\alpha_i$ or its inverse.
Using (\ref{eq:com0}) repeatedly, one can write $C$ as a length $k$
commutator, whose entries are all  products of conjugates of $\alpha_i$ or its inverse. This shows that $C\in  \Gamma_k N_i$. 
Now, using  recursively (\ref{eq:com2}) and (\ref{eq:com1}), each such commutator can be expressed as a product of length $k$ commutators, whose entries are a single conjugate of $\alpha_i$ or its inverse. 
By (\ref{eq:com3}), combined with
(\ref{eq:com1}), a length $k$ commutator can be expressed as a product of linear ones, and in our context all $k$ entries will remain a single conjugate of $\alpha_i$ or its inverse.
\end{proof}

For the next two results, let $G$ be a group which is normally
generated by $n$ elements $\alpha_1,\ldots,\alpha_n$.
For each $i\in \{1,\cdots,n\}$, denote by $N_i$ the normal subgroup of
$G$ generated by the $i$th generator.

\begin{lem}\label{lem:itsakindofbasic}
 For any $k\in\N$,
 \[
   \Gamma_{kn+1} G\subset \Gamma_{k+1} N_1\cdots  \Gamma_{k+1} N_n.
 \]
\end{lem}
\begin{proof}
Using (\ref{eq:combonus}), any element of  $\Gamma_{kn+1} G$ can be expressed as a product of length $\ge kn+1$ commutators with entries in $\big\{\alpha_j^g,\ov{\alpha}_j^g\, ; \, g\in G\big\}_j$. 
For any such commutator $C$, there exists some $i$ such that at least
$k+1$ entries of $C$ are elements of $N_i$, and Lemma \ref{lem:2.1}
implies that $C\in \Gamma_{k+1} N_i$.
 \end{proof}

Define a \emph{conjugating} endomorphism of $G$ as an endomorphism which sends each generator to a  conjugate of itself. The following is a simple adaptation of \cite[Lem.~2.37]{darne} to our setting.
\begin{lem}\label{lem:Jacko}
For all $k\ge 2$, any conjugating endomorphism $\varphi$ of $G$ induces an
automorphism of $\fract{G}{\Gamma_k G}$. In particular, if $G$ is
nilpotent, then $\varphi$ is itself an automorphism of $G$.
\end{lem}
\begin{proof}
  As a conjugating endomorphism, $\varphi$ induces the identity on
  $\fract{G}{\Gamma_{2} G}$, and more generally on 
  $\fract{\Gamma_{k-1} G}{\Gamma_{k} G}$ for all $k\geq 2$.
An induction on $k$, using the Five Lemma on the natural exact sequence
\[
  0\longrightarrow \fract{\Gamma_{k-1} G}{\Gamma_{k} G}\longrightarrow
  \fract{G}{\Gamma_{k} G}\longrightarrow \fract{G}{\Gamma_{k-1} G}
  \longrightarrow 0,
\]
then shows that $\varphi$ induces an automorphism of
$\fract{G}{\Gamma_k G}$, for any $k$. \\
In particular, if $G$ is nilpotent of order $N$, then $\varphi$ induces hence an
automorphism of $\fract G{\Gamma_N G}\cong G$.
\end{proof}

\subsection{Basic commutators}\label{sec:basic}

We now recall the notion of basic commutators in a free group, which seems to have first appeared in \cite{Phall}.
\begin{defi}\label{def:basic}
A \emph{set of basic commutators} in the set $\{\alpha_1,\cdots,\alpha_n\}$ is an infinite ordered set of commutators $\big\{C_i\big\}_i$, defined inductively as follows. 
Basic commutators of length $1$ are the $C_i=\alpha_i$ for $i=1,\cdots,n$. 
A basic commutator of length $m>1$ has form $[C_i,C_j]$, for some basic commutators $C_i,C_j$ such that \begin{itemize}
\item length$(C_i)+$length$(C_j)=m$;  
\item $C_i<C_j$, and $C_j=[C_k,C_l]$ further implies that $C_k\le C_i$.
\end{itemize}
Basic commutators of length $m$ follow those of length $<m$, and are ordered in an arbitrary fixed way with respect to each other. 
\end{defi} 

What makes this notion significant is the following fundamental result from \cite[Thm.~11.2.4]{hall}  (see also \cite[Thm.~5.13.A]{MKS}). 
\begin{theo}\label{th:hall}
 If $\big\{C_i\big\}_i$ is a set of basic commutators, and $k\ge 0$ is an integer, then any element $g$ in the free group $F_n$ can be written  in a unique way as a product 
 \[
   g = C_1^{e_1}\cdots C_{N(k)}^{e_{N(k)}} h,
   \]
 with $e_j\in \mathbb{Z}$ and $h\in \Gamma_{k+1}F_n$, where $N(k)$ is the number of basic commutators of length $\le k$. 
\end{theo}

\begin{rem}\label{rem:changement_de_bracket}
 In \cite{hall,MKS}, the convention $[a,b]=\ov{a}\ov{b}ab$ is used for commutators. 
 Although Definition \ref{def:basic} thus gives \emph{different} sets of basic commutators as the ones used in \cite{hall,MKS}, it does share the same fundamental properties, and in particular Theorem \ref{th:hall}, as well as  the coming Lemma \ref{lem:Levine}. 
 \end{rem}

\subsection{The Magnus expansion}\label{sec:magnus}

Denote by $\mathbb{Z}\langle\langle X_1,\cdots,X_n\rangle\rangle$  the ring of formal power series in the noncommuting variables $X_1,\cdots, X_n$. 
\begin{defi}
The \emph{Magnus expansion} is the group homomorphism 
\[
  E:~F_n\longrightarrow \mathbb{Z}\langle\langle
  X_1,\cdots,X_n\rangle\rangle
  \]
defined by sending $\alpha_i$ to $1+X_i$. 
\end{defi}
It is well-known that $E$ is in fact injective \cite[Thm.~5.6]{MKS}, and that it is well-behaved with respect to the lower central series, in the sense of the following fundamental result \cite{Magnus37,witt}. 
\begin{theo}\label{th:magnus} 
For all $k\ge 1$, we have $f\in \Gamma_kF_n$ if and only if $E(f)=1+ \textrm{(terms of degree $\ge k$)}$. 
\end{theo}

Basic commutators are well-behaved with respect to the Magnus expansion, in the sense of the following result, which is outlined by Levine in \cite[pp.~365]{Levine}.
\begin{lem}\label{lem:Levine}
Let $C$ be a basic commutator of length $k$, such that $\alpha_i$ occurs $r_i$ times in $C$ for each $i$.  
We have 
\[
  E(C) = 1+P+\textrm{(terms of degree $\ge k+1$)},
  \]
where $P\neq 0$ is a sum of degree $k$ terms, each involving $r_i$ times the variable $X_i$ ($i=1,\cdots,n$). 
\end{lem}
In the above, the sum $P$ of lowest degree (non trivial) terms in $E(C)$ is called the \emph{principal part} of $E(C)$. 
\begin{proof}
We proceed by induction on the length $k$, where the case $k=1$ is trivial. 
Let $C$ be a basic commutator of length $k$ for some $k>1$. There exists basic commutators $C_1$ and $C_2$, of respective length $k_1$ and $k_2$, such that $C=[C_1,C_2]$ and $k_1+k_2=k$. 
By induction hypothesis for $i=1,2$ we have that $E(C_i)=1+P_i+\textrm{(degree $>k_i$ terms)}$, with $P_i$ a sum of degree $k_i$ terms with the appropriate occurences of each variables. Thus $E(C) = 1 + P_1P_2-P_2P_1 +\textrm{(degree $>k$ terms)}$, and the conclusion would follow from the fact that $P_1P_2-P_2P_1\neq 0$. In order to see that $P_1P_2-P_2P_1$ is indeed nonzero, observe that length $k$ basic commutators form a basis for $\fract{\Gamma_kF_n}{\Gamma_{k+1}F_n}$, as a consequence of Theorem \ref{th:hall}. This in particular tells us that $C\notin \Gamma_{k+1} F_n$, hence by Theorem \ref{th:magnus} we have $P_1P_2-P_2P_1\neq 0$. 
\end{proof}

\section{Milnor invariants and $k$--reduced free action}\label{sec:12}

The main diagrammatic object of this paper will be the following. 
\begin{defi}\label{def:welded}
Consider the $2$-disk $[0,1]\times [0,1]$, equipped with fixed points $p_i \times \{\e\}$ for $i\in\{1,\cdots,n\}$ and $\e\in \{0,1\}$. 
An $n$-component \emph{welded string link} is the welded equivalence
class of an immersion of $n$ disjoint copies of the unit interval into
$[0,1]\times [0,1]$, such that the $i$th interval runs from $p_i\times
\{0\}$ to $p_i\times \{1\}$, and whose double points are decorated
either as a classical (as in usual knot diagrams) or a virtual
crossing (drawn as transverse double points); see Figure \ref{fig:moo}. 
Here, the \emph{welded equivalence} is generated by the three usual Reidemeister moves involving classical crossings, together with the OC move shown in Figure \ref{fig:moo} and the
\emph{Detour move}, which replaces an arc with only virtual crossings
(possibly none) by another arc with same endpoints and only virtual
crossings (possibly none).
\end{defi}
\begin{figure}
  \[
     \dessin{2cm}{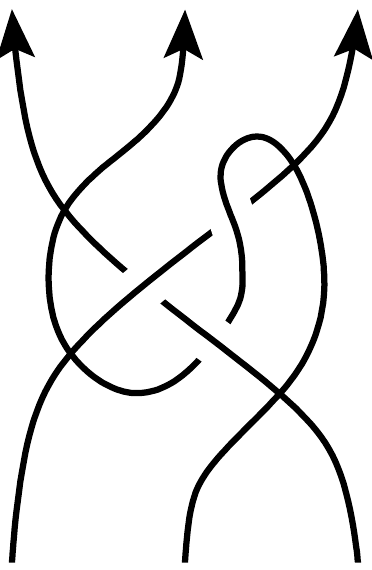}
    \hspace{2.5cm}
    \dessin{1.75cm}{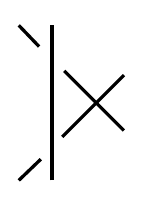}\xleftrightarrow[]{\textrm{OC}}\dessin{1.75cm}{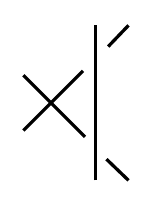}
    \]
  \caption{A $3$-component welded string link (left), 
   and the OC move (right)
  }\label{fig:moo}
\end{figure}

In particular, a welded string link without virtual crossing is merely a diagram of a classical string link. 
A key point is that classical string links inject in this way into welded string links: two diagrams without virtual crossings, that are welded equivalent, represent isotopic objects, see \cite[Thm~1.B]{GPV}. 

We denote by $\mathbf{1}$ the trivial string link diagram $\cup_i p_i\times [0,1]$.

\subsection{Brief review of welded Milnor invariants}\label{sec:milnor}

Milnor invariants are classical link invariants defined by Milnor in the fifties \cite{Milnor,Milnor2}, and extended to (classical) string links by Habbeger and Lin \cite{HL}. 
We review here the welded extension of Milnor invariants. It was given in \cite[Sec.~6]{ABMW} using a topological approach, and 
was later reformulated in \cite{MWY} in a purely %
diagrammatic way. 
 
Given a welded string link $L$, there is an associated group $G(L)$, which coincides with the fundamental group of the complement when $L$ is classical, see \cite{Kauffman}. 
As in Wirtinger's algorithm, a generator of $G(L)$ is associated with 
each arc in a diagram of $L$, where now an arc is allowed to contain virtual crossings, and each classical crossing gives a conjugacy relation: 
\[
  \dessin{1.5cm}{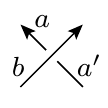}\ \leadsto\ \overline a\overline ba'b
  \hspace{1cm}
  \dessin{1.5cm}{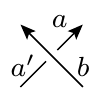}\ \leadsto\ \overline aba'\overline b.\\[-0.1cm]
\]
It is well-known that this is invariant under welded equivalence.

Denote by $L_1, \ldots, L_n$ the components of $L$. 
For each $i$, label by $a_{i,j}$ ($1\le j\le m_i+1$) the successive
arcs met along $L_i$ when following the orientation, where $m_i+1$
denotes the number of arcs of $L_i$, et by $b_{i,j}$ the generator
$a_{k,l}^{\pm1}$ associated with the overpassing arc that is met when running from $a_{i,j}$ to $a_{i,j+1}$, and the local orientation:
\[
\dessin{1.75cm}{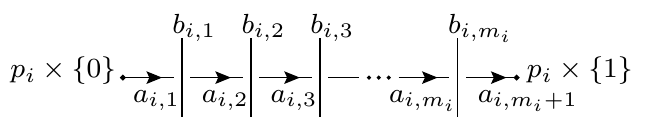}.
  \]
Then the group of $L$ has a presentation of the form 
\begin{equation}\label{eqG}
G(L) = \big\langle a_{i,j},\ 1 \leq i \leq n,\ 1 \leq j \leq m_i\ \big\vert\ \overline a_{i j+1}\overline b_{i,j}a_{i,j}b_{i,j}\ 1 \leq i \leq n,\ 1 \leq j \leq m_i-1 \big\rangle, 
\end{equation} 

\begin{defi} \label{def:longitude}
For $1 \leq i \leq n$, the \emph{preferred $i$th longitude} $\lambda_i(L)$ is given by 
$\overline a_{i,1}^{f_i}\ b_{i,1}\ b_{i,2}\ \ldots\ b_{i,m_i-1}$, 
where $f_i$ is the sum of the signs of all classical crossings involving only arcs of the $i$th component.
\end{defi}

Let us fix some integer $q\ge 2$. 
It is well-known that the images $\alpha_1, \ldots, \alpha_n$ of the generators $a_{i,1}$ in the quotient $\fract{G(L)}{\Gamma_q G(L)}$, do generate this group  \cite{Chen}. 
In particular, the image of the preferred $i$th longitude in
$\fract{G(L)}{\Gamma_q G(L)}$, can be expressed as a word, which we
still denote by $\lambda_i(L)$, in the variables $\alpha_1, \ldots, \alpha_n$.

\begin{defi}
For each sequence $I = j_1 j_2 \ldots j_l i$ of integers in $\{1,\cdots,n\}$ ($l<q$), the coefficient $\mu_L(I)$  of $X_{j_1} \cdots X_{j_l}$ in the Magnus expansion $E\big(\lambda_i(L)\big)$ is an invariant of the welded string link $L$, called a \emph{welded Milnor invariant}. 
\end{defi}
\begin{rems}\label{rem:pref}$\ $
  \begin{enumerate}
  \item These welded Milnor invariants naturally coincides with the classical construction in the case of classical string links.
  \item  The fact that $\lambda_i(L)$ represents the \emph{preferred}
    $i$th longitude implies that $E\big(\lambda_i(L)\big)$ does not contain any
    term of the form $X_i^s$, for $s\ge 1$, hence that $\mu_L(I)=0$ for
    any sequence of the form $I=ii\cdots i$.
  \end{enumerate}
\end{rems}

Recall from the introduction that, given a sequence $I$, we denote by $r(I)$ the maximum number of time that any index appears in $I$. 
For classical (string) links, Milnor invariants $\mu(I)$ with
$r(I)=1$, that is, non-repeated Milnor invariants, are known to be
link-homotopy invariants \cite{Milnor2,HL}. 
Habegger and Lin actually showed that non-repeated Milnor invariants  classify  string links up to link-homotopy \cite{HL}, a result that was later extended to the welded setting in \cite{ABMW}, where non-repeated Milnor invariants are showed to classify welded string links up to self-virtualization. 
Here, \emph{self-virtualization} is the equivalence relation generated
by the local replacement of a classical self-crossing, by a virtual
one -- what turns out to coincide with link-homotopy for classical
string links (this was implicit in \cite{ABMW} and formally stated in
\cite[Thm. 4.3]{uvw}).

\subsection{The $k$--reduced free action}\label{sec:kred}

As above, let $L$ be an $n$-component welded string link, with associated group $G(L)$.  
By definition, $G(L)$ is normally generated by the initial arcs $\alpha_i:=a_{i,1}$ of each component.  
As in the introduction, we can thus consider the \emph{$k$--reduced quotient} of $G(L)$ 
\[
  \nR_kG(L):= \fract{G(L)}{\Gamma_{k+1} N_1\cdots \Gamma_{k+1} N_n},
  \]
where $N_i$ denotes the normal subgroup generated by $\alpha_i$. 
As a consequence of Lemma \ref{lem:itsakindofbasic}, we have that the group $\nR_kG(L)$ is nilpotent of order at most $kn+1$.

 Let $F_n$ be the free group generated by $\alpha_1,\cdots,\alpha_n$.
We have the following.
\begin{lem}\label{lem:izo}
For each $k\ge 1$, we have an isomorphism 
\[
  \nR_kG(L)\simeq \nR_k F_n= \fract{F_n }{\Gamma_{k+1} N_1\cdots
    \Gamma_{k+1} N_n}.
  \]
\end{lem}
\begin{proof}
 As shown in \cite[\S 5]{ABMW} (see also \cite[\S 6.3]{arrow}), for each $q\ge 1$ we have an isomorphism 
 \[
   \fract{G(L)}{\Gamma_qG(L)}\simeq \fract{F_n}{\Gamma_{q} F_n}.
   \]
   Hence, by Lemma \ref{lem:itsakindofbasic}, we have  
 \[
   \nR_k G(L) = \nR_k\!\left(\!\fract{G(L)}{\Gamma_{kn+1} G(L)}\right)
   \simeq \nR_k\!\left(\!\fract{F_n}{\Gamma_{kn+1} F_n}\right) = \nR_k F_n.\qedhere
   \]
\end{proof}

For each $i$, consider the $i$th preferred longitude of $L$ from Definition \ref{def:longitude}. This defines an element $\lambda^k_i(L)$ in $\nR_kG(L)\simeq \nR_k F_n$, called the \emph{$k$--reduced $i$th longitude} of $L$. 

\medskip

Denote by  $\Aut_c(\nR_kF_n)$ the set of conjugating automorphisms of $\nR_kF_n$, which are automorphisms sending each generator to a conjugate of itself. 
Since $\nR_kF_n$ is nilpotent, by Lemma \ref{lem:Jacko}, any conjugating endomorphism of $R_k F_n$ is in $\Aut_c(\nR_kF_n)$, for all $k\ge 1$.  

\begin{defi}\label{def:action}
The \emph{$k$--reduced free action} associated with $L$, 
\[
  \varphi^{(k)}_L\in \Aut_c(\nR_kF_n),
  \]
is defined by sending each generator $\alpha_i$ to its conjugate by $\lambda^k_i(L)$. 
\end{defi}

This action can be seen as a generalization of Habegger-Lin's representation for the group of link-homotopy classes of string links \cite[Thm.~1.7]{HL}, 
in the sense that the case $k=1$ recovers their construction. 

\subsection{Milnor invariants and $k$--reduced free action} \label{sec:1>2} 

The main purpose of this section is Theorem \ref{thm:12}, which implies the equivalence   (1)$\Leftrightarrow$(2) in our main theorem. 

Throughout the rest of this paper, we shall make use of the following. 
\begin{nota}
Recall that $N_i$ denotes the normal subgroup of $F_n$ generated by
the $i$th generator $\alpha_i$ ($i\in \{1,\cdots,n\}$). 
For $k\ge 1$, we set
\[
  J^k:=\Gamma_{k+1}N_1\cdots \Gamma_{k+1}N_n
 \,\,\, \textrm{ and }\,\,\, 
  J_i^k:=\Gamma_{k+1}N_1\cdots \Gamma_{k}N_i\cdots \Gamma_{k+1}N_n.
\]
Denote also by $R^k$ the two-sided ideal of $\mathbb{Z}\langle\langle
X_1,\cdots,X_n\rangle\rangle$ generated by all terms having at least
$k+1$ occurences of some variable, and by $R^k_i$ the ideal generated by terms having either at least $k+1$ occurences of some variable, or $k$ occurences of the variable $X_i$.  
\end{nota}

\begin{theo}\label{thm:12}
 Let $L$ and $L'$ be two $n$-component welded string links. The following are equivalent: 
 \begin{enumerate}
  \item[(i)] $\mu_L(I)=\mu_{L'}(I)$ for any sequence $I$ with $r(I)\le k$.  
  \item[(ii)] the $k$--reduced $i$th longitudes $\lambda^k_i(L)$ and $\lambda^k_i(L')$ are congruent modulo $J_i^k$, for all $i$.  
  \item[(iii)] $L$ and $L'$ induce the  same $k$--reduced free action $\varphi_L^{(k)}=\varphi_{L'}^{(k)}\in \Aut_c(\nR_kF_n)$. 
\end{enumerate}
\end{theo}
\medskip 

The rest of this section is devoted to the proof of Theorem \ref{thm:12}, which is done in three steps. 
\medskip 

Firstly, we prove that (ii)$\Rightarrow$(iii). 
Suppose that the $k$--reduced $i$th longitudes $\lambda^k_i(L)$ and $\lambda^k_i(L')$ are congruent modulo $J_i^k$, for all $i$. 
We have $ \lambda^k_i(L) \overline{\lambda^k_i(L')}=X g_i$,
for some $g_i\in \Gamma_kN_i$ and some $X\in \prod_{j\neq i} \Gamma_{k+1}N_j$.
By (\ref{eq:com2}) we thus have that 
\[
  \big[\lambda^k_i(L)\overline{\lambda^k_i(L')},\alpha_i\big] =
  [g_i,\alpha_i]^{\ov{X}}[X,\alpha_i] \in J^k.
  \]
This readily  implies that $\big[\alpha_i,\lambda^k_i(L) \big]\equiv \big[\alpha_i,\lambda^k_i(L')\big]$ mod $J^k$, which is equivalent to saying that $L$ and $L'$ induce the same $k$--reduced free action: $\varphi_L^{(k)}=\varphi_{L'}^{(k)}\in \Aut_c(R_{k}F_n)$.  
\medskip 

Secondly, we prove that (iii)$\Rightarrow$(i). 
If $\varphi_L^{(k)}=\varphi_{L'}^{(k)}\in \Aut_c(R_{k}F_n)$, then for each $i$ we have $\alpha_i^{\lambda^k_i(L)}\equiv \alpha_i^{\lambda^k_i(L')}$ mod $J^k$, which is equivalent to 
\[
  \alpha_i
  \big(\lambda^k_i(L)\overline{\lambda^k_i(L')}\big)\overline{\alpha}_i\equiv
  \lambda^k_i(L)\overline{\lambda^k_i(L')} \textrm{ mod $J^k$.}
  \]
Taking the Magnus expansion then gives  
\begin{equation}\label{eq:Magnus}
 X_i E \big (\lambda^k_i(L)\overline{\lambda^k_i(L')}\big)  - E \big (\lambda^k_i(L)\overline{\lambda^k_i(L')}\big) \big (X_i-X_i^2+\cdots \big) - X_i  E \big (\lambda^k_i(L)\overline{\lambda^k_i(L')}\big)\big(X_i-X_i^2+\cdots \big) \in R^k. 
 \end{equation}

If $E \big(\lambda^k_i(L)\big)-E \big(\lambda^k_i(L')\big)$ lives in $R^k_i$, then $L$ and $L'$ have same 
Milnor invariants $\mu(I)$ with $r(I)\le k$. 
Suppose by contradiction that 
\[
  E \big(\lambda^k_i(L) \big)-E \big(\lambda^k_i(L')\big) \equiv S_q +
  \left(\textrm{terms of degree $>q$}\right)  \textrm{ mod $R^k_i$},
  \]
for some $q$ and a sum $S_q$ of degree $q$ terms which are \emph{not} in $R^k_i$. 
Multiplying the above by $E \big(\overline{\lambda^k_i(L')}\big)$,
which by definition has constant term equal to $1$, we have  
\[
  E \big(\lambda^k_i(L)\overline{\lambda^k_i(L')}\big) \equiv 1+ S_q +
  \left(\textrm{terms of degree $>q$}\right)  \textrm{ mod $R^k_i$}.
  \]
Equation  (\ref{eq:Magnus}) then gives 
\[
  X_i E \big(\lambda^k_i(L)\overline{\lambda^k_i(L')}\big)  - E \big(\lambda^k_i(L)\overline{\lambda^k_i(L')}\big)  X_i\equiv 
 X_i S_q - S_q X_i + \left(\textrm{terms of degree $>q+1$}\right) 
 \textrm{ mod $R^k_i$}.
 \]
Observe that $X_i S_q$ and $S_q X_i$ are not in $R^k$, since $S_q\notin R_i^k$.
Hence, if $X_i S_q\neq S_q X_i$, we obtain a contradiction with (\ref{eq:Magnus}). 
But if  $X_i S_q=S_q X_i$, a simple combinatorial argument shows that $S_q$ can be nothing else than the monomial $X_i^q$. 
This tells us that $E \big(\lambda^k_i(L) \big)-E
\big(\lambda^k_i(L')\big)$ contains the term $X_i^q$, which is in
contradiction with the fact that $\lambda^k_i(L)$ and $\lambda^k_i(L')$
are images of \emph{preferred} $i$th longitudes, see Remark \ref{rem:pref}~(ii). 
\medskip 

 Thirdly and lastly, let us prove that (i)$\Rightarrow$(ii). 
For this purpose, we prove the following, which should be thought of as a \quote{$k$--reduced} version of the injective property of the Magnus expansion (Theorem~\ref{th:magnus}). 
\begin{prop}\label{th:kredinjectivity}
Let $g$ be an element of $F_n$. Then $g\in J_i^k$ if and only if $E(g)\in 1+R_i^k$. 
\end{prop}

Assuming this result momentarily, we immediately deduce the desired implication (i)$\Rightarrow$(ii) of Theorem \ref{thm:12}. 
Indeed, consider two welded string links  $L$ and $L'$ with same 
Milnor invariants $\mu(I)$ with $r(I)\le k$. 
This precisely means that, for each $i$, their respective $i$th preferred longitudes $\lambda^k_i(L)$ and $\lambda^k_i(L')$ satisfy 
\[
  E \big(\lambda^k_i(L') \big) \equiv  E \big(\lambda^k_i(L) \big)
  \textrm{ mod $R_i^k$}.
  \]
This rewrites as $E \big (\overline{\lambda^k_i(L)}\lambda^k_i(L') \big)=1+R_i^k$, which by Proposition \ref{th:kredinjectivity} gives that $\overline{\lambda^k_i(L)}\lambda^k_i(L')\in J_i^k$, as desired.
\medskip

\begin{proof}[Proof of Proposition \ref{th:kredinjectivity}]
The \quote{only if} part of the statement follows easily from well-known properties of the Magnus expansion \cite{MKS}, so we prove here the other implication. 
By Theorem \ref{th:hall},  there is a set of basic commutators $\{C_i\}_i$ such that
\[
  g = C_1^{e_1}\cdots C_{N(nk)}^{e_{N(nk)}} h,
  \]
for some unique integers $e_1,\cdots,e_{N(nk)}\in \mathbb{Z}$ and a unique element $h\in \Gamma_{nk+1}F_n$.
Set $g_j:=\prod_{\textrm{length$(C_{j_m})=j$}} C_{j_m}^{e_{j_m}}$, so that $g = g_1\cdots g_{kn} h$. 
Since $h\in J_i^k$, it remains to show that $g_j\in J_i^k$ for all $j$.  
Suppose by contradiction that this is not the case, and let $p$ be the smallest integer such that 
$g_p=\prod_{m=1}^N C_{p_m}^{e_{p_m}}$ is not in $J_i^k$ ($N\ge 1$). 
This means that there is a nonempty subset $S$ of $\{1,\cdots,N\}$ such that $e_{p_l}\neq 0$ and $C_{p_l}\notin J_i^k$, for all $l\in S$.
\\
Lemma \ref{lem:Levine} tells us that, for each $m$, the principal part of $E(C_{p_m})$ is either in $R_i^k$, or is in the $\mathbb{Z}$-module generated by monomials with $<(k-1)$ copies of $X_i$ and $<k$ copies of any other variable.
This implies in particular that the sum of the principal parts of $E \big (\prod_{l\in S} C_{p_l}^{e_{p_l}}\big)$ is not in $R_i^k\setminus \{0\}$. 
But as a property of basic commutators, we know that these principal
parts are linearly independent \cite{MKS}; it follows that $E(g_p)-1$ is non
trivial and not in $R_i^k$. 
By minimality of $p$, this implies that $E(g)-1$ is not in $R_i^k$. 
This is a contradiction, which concludes the proof.
\end{proof}

\begin{rem}
The case $k=1$ of Proposition \ref{th:kredinjectivity} was previously established in \cite[Prop.~7.10]{ipipipyura}.
One can actually further generalize this result as follows. 
Let $\mathbf{r}=(r_1,\cdots,r_n)$ be an $n$-tuple of positive integers. 
Consider the subgroup $J^\mathbf{r}:=\Gamma_{r_1}N_1\cdots \Gamma_{r_n}N_n$ of $F_n$, and the ideal $R^\mathbf{r}$ of $\mathbb{Z}\langle\langle X_1,\cdots,X_n\rangle\rangle$ generated by all terms where the variable $X_i$ appears at least $r_i$ times, for $i=1,\cdots,n$.  
Then the above proof generalizes in a straightforward way to show that
an element $g\in F_n$ is in $J^\mathbf{r}$ if and only if $E(g)$ is in $1+R^\mathbf{r}$.
\end{rem}

\section{Arrow calculus and self $w_k$-concordance}\label{sec:weldedwk}

\subsection{Arrow calculus for welded objects}\label{sec:arrow}

Let us briefly review arrow calculus, which is diagrammatic calculus developped in \cite{arrow} for the study of  welded objects.
In Section \ref{sec:formalism}, we explain how this diagrammatic tool is intimately related to the commutator
calculus reviewed in Section \ref{sec:1}.

Let $L$ be an $n$-component welded string link.

\begin{defi}
A \emph{$w$-tree} for $L$ is a planar immersion of an oriented, connected uni-trivalent tree $T$, such that
\begin{itemize}
	\item the set of all vertices of $T$ is embedded in the interior of $[0,1]\times [0,1]$, such that trivalent vertices are disjoint from $L$ and univalent vertices are in $L\setminus \{\textrm{crossings of $L$}\}$;
	\item at each trivalent vertex, the three incident edges are cyclically ordered, and there are two ingoing and one outgoing edge; 
	\item edges of $T$ may cross virtually $L$ or $T$ itself;
	\item edges are possibly decorated by some \emph{twists} $\bullet$, which are disjoint from all crossings and vertices, and which satisfy the rule $\,\,\includegraphics[scale=1]{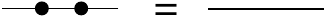}$. 
\end{itemize}
A univalent vertex of $T$ is a \emph{head} if $T$ is locally oriented towards $L$, and it is a \emph{tail} otherwise. 
For a union of $w$-trees for $L$, we allow virtual crossings among
edges, and we require all vertices to be pairwise disjoint. See Figure \ref{fig:wTree} for an example.
\end{defi}
We note that a $w$-tree contains a single head, due to the orientation convention at trivalent vertices.  

\begin{figure}
  \[
\dessin{3cm}{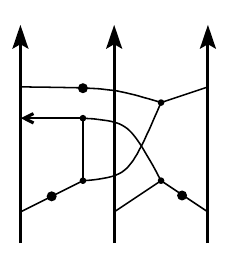}
    \]
  \caption{An example of $w$-tree for the trivial $3$--component
    string link}
  \label{fig:wTree}
\end{figure}

\begin{defi}\label{def:lineartree}
A $w$-tree is \emph{linear}, if it has the following shape as an abstract tree    
\[
  \textrm{\includegraphics[scale=0.7]{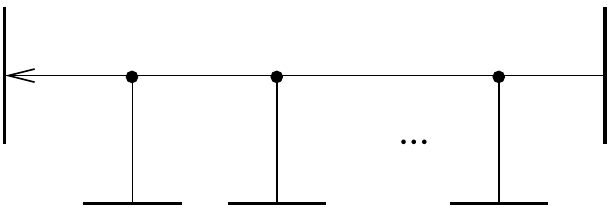}}
  \]
with possibly a number of $\bullet$ on its edges. 
\end{defi}

\begin{defi}
The \emph{degree} of $T$ is its number of tails or, equivalently, half the total number of vertices. 
A degree $k$ $w$-tree is called a \emph{$w_k$-tree}, and a $w_1$-tree is also called a $w$-arrow.  
\end{defi}

Given a union of $w$-arrows $A$ for $L$, there is a notion of  \emph{surgery along $A$}, which produces a new welded string link $L_A$ as follows:  
\[
  \textrm{\includegraphics{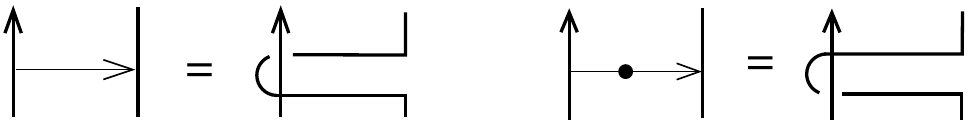}. }
\]
In general, if $A$ contains some virtual crossing, this likewise introduces pairs of virtual crossings in $L_A$.

Now, given a union of $w_k$-trees $P$ for $L$, one can define the notion of surgery along $P$ by first replacing $P$ by a union of  $w$-arrows, called the \emph{expansion} of $P$ and denoted by $E(P)$, defined recursively by the local rule illustrated below, then performing surgery on  $E(P)$. The result will be denoted by $L_P$.
\[
  \textrm{\includegraphics[scale=0.9]{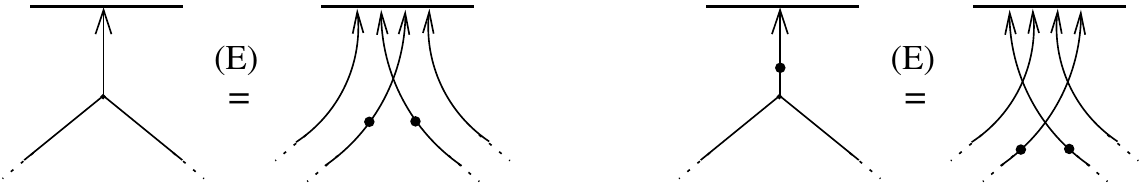}.}
  \]
\begin{rem}
The above figure suggests that the union $E(P)$ of $w$-arrows obtained
from $P$ by applying (E) recursively, has the shape of an
\quote{iterated commutator} of $w$-arrows. 
This observation is made rigourous and further discussed in Section \ref{sec:formalism}.
\end{rem}

\begin{defi}\label{def:wTree}
A \emph{$w$-tree presentation} $(\mathbf{1}, P)$ for $L$ is a union $P$ of $w$-trees for the trivial diagram $\mathbf{1}$, such that $L = \mathbf{1}_P$. 
Two arrow presentations $(\mathbf{1},P)$ and $(\mathbf{1},P')$ representing equivalent welded string links are called \emph{equivalent}. 
\end{defi}

The main point of this notion is that any welded string link admits a $w$-tree presentation \cite[Prop.~4.2]{arrow}. 
Moreover, in \cite[Thm.~5.21]{arrow}, a set of moves on $w$-trees is provided, which suffice to deform any $w$-tree presentation of $L$ into any other. 
These moves imply further operations, hence a full diagrammatic calculus called \emph{arrow calculus}, which can be used to study welded objects and their invariants. 
We do not reproduce all these moves here, but only provide below those that will be needed in this paper : 
\[ \textrm{\includegraphics[scale=0.8]{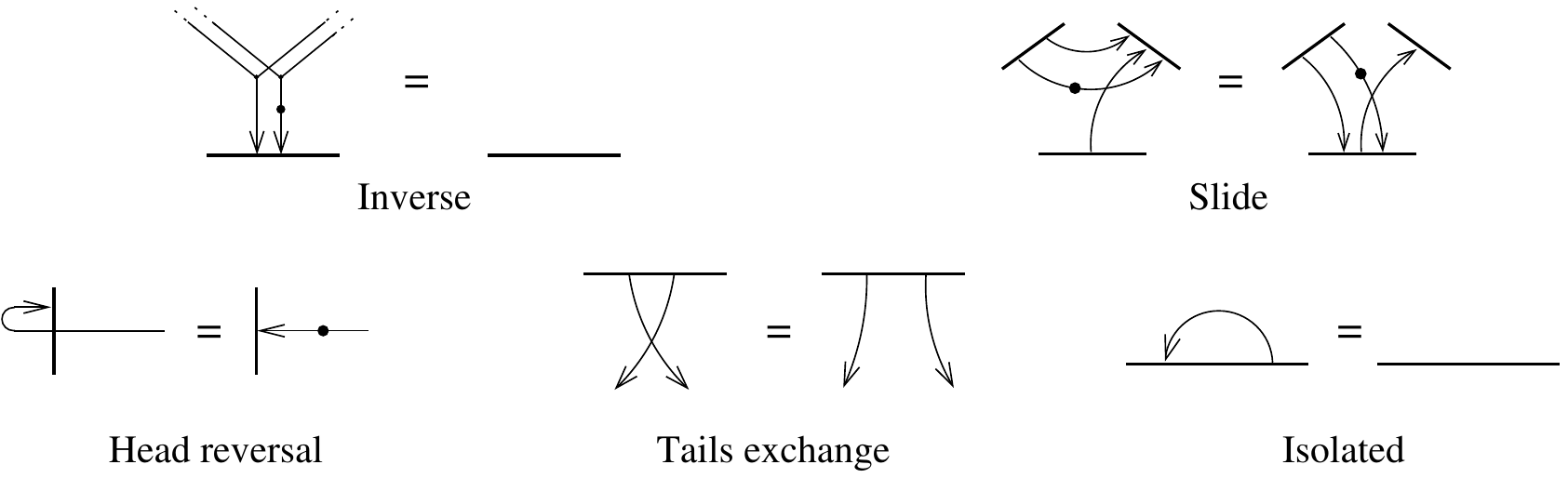}}\]
We refer the reader to \cite{arrow} for more details on arrow calculus.  

\medskip

We will also use the following, which refines the notion of
equivalence given in Definition \ref{def:wTree}.
\begin{defi}
Let $(\mathbf{1},P)$ be a $w$-tree presentation for $L$, and let $S$ be a subset of $P$. Denote by $\sigma\subset \1$ a neighborhood of the endpoints of $S$, which identifies with a trivial diagram, such that $\sigma$ is disjoint from $P\setminus S$. 
Let $S'$ be a union of $w$-trees for $\sigma$. 
We say that $S$ and $S'$ are \emph{locally equivalent} whenever $\sigma_{S'}$ is welded equivalent to $\sigma_S$. 
\end{defi}
Note that, in this setting, $(\mathbf{1},P)$ and $\big(\mathbf{1},(P\setminus S)\cup S' \big)$ are equivalent $w$-tree presentations. 

\subsection{Algebraic formalism for $w$-trees}\label{sec:formalism}
  
Let $(\1,P)$ be a $w$-tree presentation for a welded diagram. 

Let $A$ be a union of $w$-arrows in $P$ that are \emph{adjacent}, in
the sense that  \emph{all} heads in $A$ are met consecutively on a
portion $h_A$ of $\1$, called the \emph{support} of $A$, without
meeting any crossing or endpoint.\footnote{We shall also say in this situation that the heads of $A$ are \emph{adjacent}.}
The \emph{complement} of $A$ is then defined as the $w$-tree presentation $(\1,P\setminus A)$ obtained from $(\1,P)$ by removing all $w$-arrows in $A$. 
The support $h_A$ and the arrow heads of $P\setminus A$ then cut the strands of $\1$ into intervals that we label by letters $x_1,\ldots,x_p$; we set $F_{(P;A)}:=\langle x_1,\ldots,x_p\rangle$, the free group generated by these letters. 

\begin{figure}
  \[
\dessin{3cm}{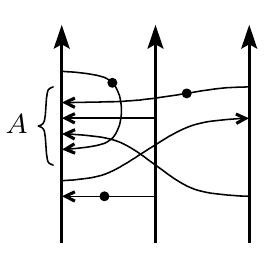}\ \leadsto\ \dessin{3cm}{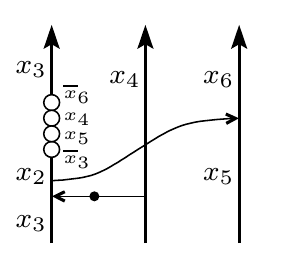}\ \leadsto\ \ w(A)=\overline
x_3x_5x_4\overline x_6
    \]
  \caption{An example of word associated to a set of adjacent
    $w$-arrows}
  \label{fig:Cut}
\end{figure}

Assuming that all arrow heads are met to the right side when traveling
along $h_A$ following its orientation (this is always possible thanks
to the Head reversal move),  each head of an arrow in $A$ can be
labeled\footnote{It should be noted that replacing arrows by
    labels corresponds actually to the \emph{cut-diagram} point of
    view on welded objects, introduced in \cite{AMY}. } by the letter
$x_i$ or $\ov{x}_i$, depending on whether the arrow contains an even or an odd number of twists, where $x_i$ is the label of the interval on which the tail lies. 
We can then define an element $w(A)\in F_{(P;A)}$  by reading the
labels in order when running along $h_A$ according to its orientation;
see Figure \ref{fig:Cut} for an example. 
Notice that this \emph{word} $w(A)$ remains unchanged under a Tails exchange move and that, conversely, it  determines the union of $w$-arrows $A$ up to Tails exchange moves.  
More generally, we have the following. 
\begin{lem}\label{lem:word}
Suppose that two $w$-tree presentations $(\1,P)$ and $(\1,P')$ only differ by replacing a union of adjacent arrows $A$ by another union of adjacent arrows $A'$ with same support. 
Then $F_{(P;A)}=F_{(P';A')}$, and $w(A)=w(A')$ if and only if $A$ and $A'$ differ by a sequence of Inverse moves and Tails exchange moves. 
\end{lem}
\begin{proof}
 First note that $A$ and $A'$ have same complement, hence we indeed have that $F_{(P;A)}=F_{(P';A')}$. 
 Since this is a free group, we have $w(A)=w(A')$ if and only if these two words differ by inserting or deleting copies of $x_j \ov{x}_j$ or $\ov{x}_j x_j$ for any $j$. But this is equivalent to saying that $A$ and $A'$ differ by a sequence of Inverse moves (with pairs of $w$-arrows)  and Tails exchange moves.  
 \end{proof}
 
Lemma \ref{lem:word} provides a one-to-one correspondence between sets of adjacent arrows with a fixed support and complement, up to equivalence, and elements in the associated free group. 
Under this correspondence, our conventions for the commutator notation
$[x,y]$ and the conjugation notation $x^y$, given in Section \ref{sec:LCS}, have natural diagrammatic counterparts. 

It is easily observed that a commutator corresponds to the expansion of a $w$--tree:
\[
x\overline y\overline x y=[x,y]\quad \leftrightarrow\quad  \dessin{1.5cm}{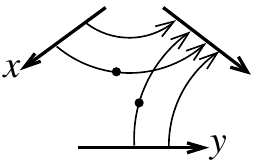}\, =\, \dessin{1.5cm}{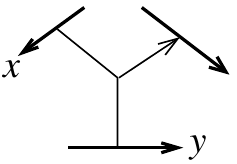}.
  \]
In general, to a $w$-tree $T$ which is part of a $w$-tree presentation $(\1,P)$, corresponds a word $w(T)\in F_{\left(P;E(T) \right)}$, which is defined as the word corresponding to its expansion; in other words, we set $w(T)=w \big(E(T) \big)$. 
From the definition, the word $w(T)$ can be directly read from $T$ using the following procedure.  
Label each edge of $T$ containing a tail by the generator $x_i$ at the tail, then  
label the remaining edges of $T$ by recursively applying the local
rules of Figure \ref{fig:w(T)}; the label at the edge containing the
head is $w(T)$; see Example \ref{ex:conjugate}~(i).
\begin{figure}
\begin{center}
   \includegraphics[scale=0.8]{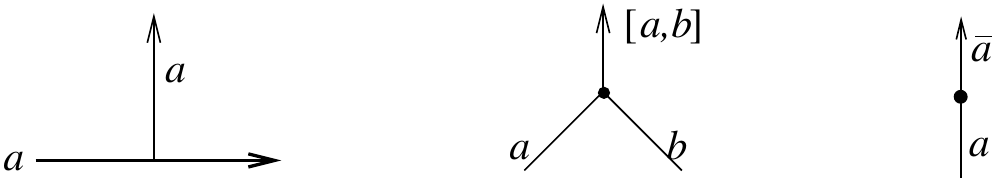}
  \caption{Procedure to compute $w(T)$ from a $w$-tree}\label{fig:w(T)}
\end{center}
\end{figure}

Note that, under this correspondence, the word associated to a linear $w$-tree (Definition \ref{def:lineartree}), is a linear commutator with entries in $\{x_i,\ov{x}_i\}_i$, in the sense of Definition \ref{def:linear}.

The conjugation notation, on the other hand, corresponds to a situation where the Slide move can be performed:
\[
\overline y x y=x^y\quad \leftrightarrow\quad \dessin{1.5cm}{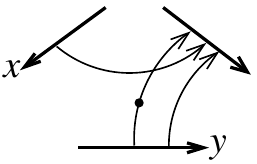}\, =\, \dessin{1.5cm}{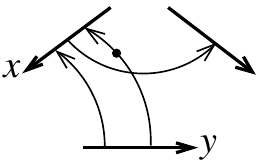},
  \]
Since two $w$-tree presentations that differ by a Slide move are locally equivalent, the above rightmost picture can be seen as the diagrammatic counterpart of notation $x^y$ in our correspondence.

\subsubsection{Conjugated trees}

Combining these two observations, we can thus extend our correspondence to a wider range of $w$-tree presentations: 
\begin{defi}\label{def:conjugate}
A $w$-tree $T$ for $\1$ is called \emph{conjugated} if there exists a union $U$ of \emph{pairs of conjugating $w$-arrows} for $T$.  Here, a pair of conjugating $w$-arrows for $T$ is a union of two parallel $w$-arrows that only differ by a twist, and whose heads are on the same component of $\1$ and only separated by a tail of $T$, and possibly other nested pairs of conjugating $w$-arrows. See Example \ref{ex:conjugate}. 
\end{defi}

Let $T$ be a conjugated $w$-tree $T$ with union $U$ of conjugating $w$-arrows, in a $w$-tree presentation $(\1,P)$.
One can define a corresponding word $w^U(T)\in F_{(P\setminus U;E(T))}$,
which is the word associated with the union of adjacent $w$-arrows
obtained by taking the expansion of $T$ and applying Inverse moves and Slide moves\footnote{Since each tail of $T$ yields a \emph{union}
    of $w$-arrows after expansion, one first need to use the Inverse
    move several times between these tails to be able to perform the Slide moves.} to $w$-arrows in $U$.
This word $w^U(T)$ can be directly read off $T\cup U$, by substituting each letter $a$ in $w(T)\in F_{(P\setminus U;E(T))}$ by its conjugate $a^{w(X)}$ whenever the corresponding tail of $T$ admits a union of conjugating $w$-arrows $X\cup \ov{X}\subset U$, where $X$ denotes the union of $w$-arrows met before the tail according to the orientation (and $\ov{X}$ are the $w$-arrows met after the tail). 
Let us illustrate this concretely on an example:  

\begin{exemple}\label{ex:conjugate}
 The $w_4$-tree $T$ shown below is a conjugated tree (here, the labels $a$, $b$, $c$, $d$ and $e$ are not necessarily mutually distinct).   
 \[
   \textrm{\includegraphics[scale=0.9]{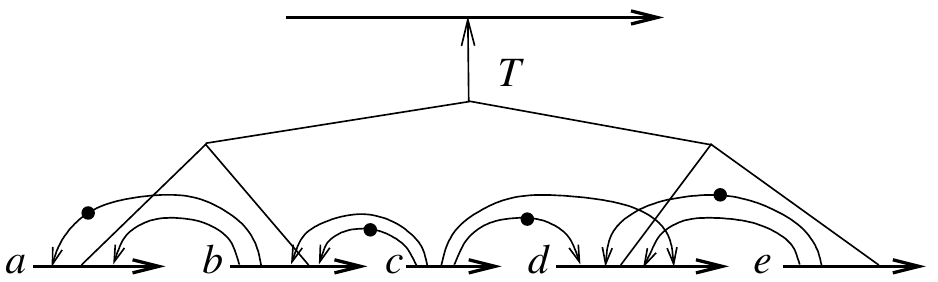}}
   \]
 
 (i)~ Ignoring all conjugating $w$-arrow, we have that $w(T)=\big[[a,b],[d,e]\big]$. 
 
 (ii)~ Denoting by $U$ the union of conjugating $w$-arrows for $T$, we have 
 \[
   w^U(T) = \big[[a^{b},b^{\ov{c}}],[d^{e c},e]\big].
   \]
\end{exemple}

\begin{rem}
  We stress that, for a $w_k$--tree $T$, $w^U(T)$ is a 
  length $k$ commutator whose entries are conjugates of $\big\{x_i,\ov{x}_i\big\}_i$. 
Conversely, it follows from the above that for any length $k$ commutator $w\in F_n$  whose entries are conjugates of $\{x_i,\ov{x}_i\}_i$, there exists a $w_k$-tree $T$ and a union $U$ of conjugating arrows such that $w=w^U(T)$.
This applies more generaly to \emph{products} of such commutators, which then correspond to \emph{adjacent} conjugated trees, \emph{i.e.} conjugated trees with adjacent heads. 
\end{rem}

\begin{rem}\label{rem:conjugated}
 Deleting a conjugated tree $T$, in the notation of Definition \ref{def:conjugate}, yields the union $U$ of conjugating $w$-arrows, which can in turn all be deleted by using the Inverse move.  
 \end{rem}

For this paper, we will need the following, which is a direct
translation of Lemma \ref{lem:2.1} via the above correspondence: 
\begin{lem}\label{cor:2.1}
Let $T$ be a $w_l$-tree in some union of $w$-trees for $\1$, 
such that the $i$th component of $\1$ contains $k$ tails of $T$ ($k\le l$). 
Then $T$ is locally equivalent to a union of adjacent conjugated linear $w_k$-trees, with all tails on the $i$th component.  
\end{lem}

\subsubsection{Relation to the welded group}\label{sec:weldedgroup}

Finally, let us recall how, given a $w$-tree presentation $(\mathbf{1},P)$ for $L$, one can associate a presentation for the group $G(L)$ defined in Section \ref{sec:milnor}, using our algebraic formalism. 

Again, by the Head reversal move, one can freely assume that all heads of $P$ are attached to the right side of $\1$ according to the orientation.  
The heads of $P$ split $\mathbf{1}$ into a union of arcs, each of which yields a generator $x_i$, and we denote by $F_P$ the free group generated by $\big\{x_i \big\}_i$. 

Now, let $T$ be a single $w_k$-tree in $P$, and denote by $a$ and $b$
the two generators of $F_P$ associated with the arcs to the left and right of its head, respectively. 
Following the above, a word $w(T)$ is defined by expanding $T$ into a union $E(T)$  of adjacent $w$-arrows, and writing the associated word.  
This $w(T)$ naturally lives in $F_P$, since the complement of $h_{E(T)}$ is obtained by just deleting a neighborhood of the head of $T$. 

Figure \ref{fig:wirtinger2} then illustrates how $T$ yields a conjugation relation $R_T$:$~b=\ov{w(T)}aw(T)$ among the two generators $a,b$ separated by its head. 
\begin{figure}[!h]
\begin{center}
   \includegraphics[scale=0.9]{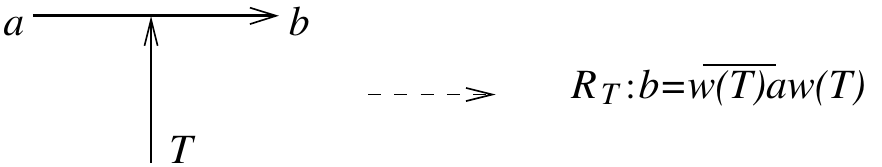}
  \caption{Conjugating relation $R_T$ associated with $T$}\label{fig:wirtinger2}
\end{center}
\end{figure}

In fact, we obtain in this way the following presentation (see \cite[\S~6.1.1]{arrow}): 
\[
  G(L) = \langle \{x_i\}_i \,\vert \, \{R_T\}_{T\in P} \rangle.
  \]
In particular, the $i$th preferred longitude of $L$ from Definition \ref{def:longitude}, can be written as $\lambda_i(L)=\alpha_i^{s_i}\prod_{T\in P_i} w(T)$, for some $s_i\in \mathbb{Z}$, where $P_i$ is the subset of $w$-trees in $P$ whose heads are on the $i$th component of $\1$, ordered according to their occurence on the $i$th component when following the orientation. 

\begin{rem}\label{rem:conj1}
 Observe that, in the notation of Figure \ref{fig:wirtinger2}, if the head of $T$ is on the $i$th component then $a,b\in N_i$, and if moreover $w(T)\in \Gamma_k N_i$ for some $k$, then we have: $b=\ov{w(T)}aw(T) = a \big[\ov{a},w(T) \big]\equiv a \textrm{ mod $\Gamma_{k+1} N_i$}$.
\end{rem}

\subsubsection{Self $w_k$-equivalence and $w^{(k+1)}$-equivalence}\label{sec:wk}

This section is concerned with the following families of equivalence relations for welded objects. 
\begin{defi}\label{def:selfwk}
Let $k\ge 1$. The \emph{$w_k$-equivalence}, resp. \emph{self $w_k$-equivalence}, is the equivalence relation generated by welded equivalence and surgery along $w$-trees, resp. self $w$-trees, of degree $\ge k$. 
Here, a \emph{self $w$-tree} is a $w$-tree whose endpoints all lie on a same component. 
The \emph{$w^{(k+1)}$-equivalence} is the equivalence relation generated by welded equivalence and surgery along $w$-trees having at least $k+1$ endpoints on a same component. 
\end{defi}

\begin{rem}\label{rem:repeat}
Given a $w_{kn+1}$-tree $T$ for some $n$-component string link, there
is necessarily some index $i$ such that $T$ 
has at least $k+1$ ends on the $i$th component. 
This elementary observation shows that  the $w_{kn+1}$-equivalence
implies the $w^{(k+1)}$-equivalence. 
\end{rem}

The following will play a central role in proving our main theorem, but might also be of independent interest in the study of arrow calculus.  
\begin{theo}\label{th:nonobvious}
 Two welded string links are self $w_k$-equivalent, if and only if they are $w^{(k+1)}$-equivalent.
\end{theo}

\begin{proof}
 Since a self $w$-tree of degree $\ge k$ has at least $k+1$ endpoints, which are all attached to a same component, we clearly have the \quote{only if} part of the statement. 

In order to prove the \quote{if} part, consider a $w$-tree $T'$ for an
$n$-component welded string link $L$, with $k+1$ endpoints on the
$i$th component $L_i$ of $L$.  We distinguish two cases.
If the head of $T$ is on $L_i$, then by Lemma \ref{cor:2.1}, it is locally equivalent to a union of conjugated self $w_k$-trees on the $i$th component. 
Hence in this case, $L_T$ is clearly self $w_k$-equivalent to $L$ by Remark \ref{rem:conjugated}. 
Now, if the head of $T$ is on the $j$th component $L_j$ of $L$ for some $j\neq i$, then Lemma \ref{cor:2.1} more precisely tells us that $T$ is locally equivalent to a union of conjugated linear $w_{k+1}$-trees, with all tails on the $i$th component and with head on the $j$th component. 
Consider one such linear $w$-tree, as on the left-hand side of the
figure below: 
\[ \includegraphics[scale=0.69]{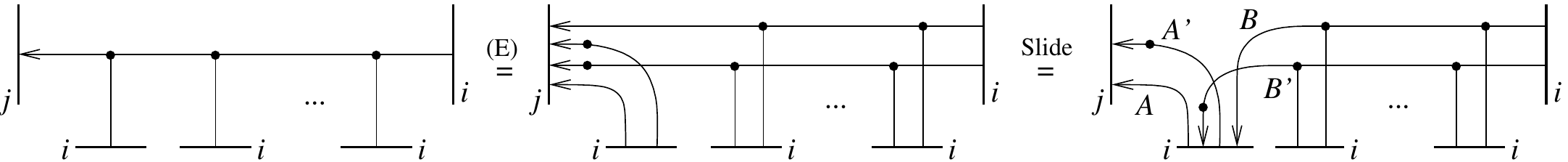} \]
This figure shows how applying the expansion move (E) to such a tree, followed by the Slide move,  yields a union of two $w$-arrows $A\cup A'$ and two self $w_{k}$-trees $B\cup B'$ on the $i$th component. Deleting $B\cup B'$  up to self $w_k$-equivalence, then deleting $A\cup A'$ by the Inverse move, yields the empty diagram. 
This observation, together with Remark \ref{rem:conjugated}, 
shows that $L_T$ is self $w_k$-equivalent to $L$. 
\end{proof}

\begin{rem}
  The argument of this proof can be used to show that the (self) $w_k$-equivalence is generated by welded equivalence and surgery along (self) $w_k$-trees, rather than (self) $w$-trees of degree $\ge k$. 
\end{rem}

\subsection{Self $w_k$-concordance and $w^{(k+1)}$-concordance} 
\label{sec:proof13}

The main purpose of this section is Theorem \ref{th:23}, which implies the equivalence   (1)$\Leftrightarrow$(3) in our main theorem. 
First, we introduce the self $w_k$-concordance relation, along with the $w^{(k+1)}$-concordance equivalence.  
\medskip 

The notion of concordance for welded links was introduced and studied in \cite{BC,Gaudreau}.

Two $n$-component welded string links $L$ and $L'$ are \emph{welded concordant} if one can be obtained from the other by a sequence of welded equivalence and the  birth/death and saddle moves of Figure \ref{fig:concmoves}, such that, for each $i\in\{1,\cdots,n\}$, the number of birth/death moves used to deform the $i$th component of $L$ into the $i$th component of $L'$ is equal to the number of saddle moves. 
\begin{figure}[!h]
 \includegraphics[scale=0.9]{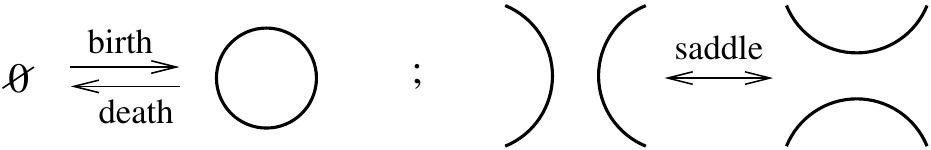}
 \caption{The birth/death and saddle moves}\label{fig:concmoves}
\end{figure}

\begin{defi}\label{def:selfwkconc}
Let $k\ge 1$. The  \emph{$w_k$-concordance}, resp.  \emph{self $w_k$-concordance}, is the equivalence relation generated by welded concordance and $w_k$-equivalence, resp. self $w_k$-equivalence.  
The \emph{$w^{(k+1)}$-concordance} is the equivalence relation generated by $w^{(k+1)}$-equivalence and welded concordance. 
\end{defi}

The following is the main result of this section. 
\begin{theo}\label{th:23}
 Let $L$ and $L'$ be two $n$-component welded string links. The following are equivalent: 
 \begin{enumerate}
  \item[(i)] $\mu_L(I)=\mu_{L'}(I)$ for any sequence $I$ with $r(I)\le k$. 
  \item[(ii)] $L$ and $L'$ are $w^{(k+1)}$-concordant. 
  \item[(iii)] $L$ and $L'$ are self $w_k$-concordant. 
\end{enumerate}
\end{theo}

The rest of this section is devoted to the proof of Theorem \ref{th:23}.  
\medskip

We already have directly that (ii) and (iii) are equivalent, by Theorem \ref{th:nonobvious}. 
In the rest of the proof, we use the fact from Theorem \ref{thm:12} that (i) is equivalent to saying that the $k$--reduced $i$th
longitudes $\lambda^k_i(L)$ and $\lambda^k_i(L')$ are congruent modulo $J_i^k=\Gamma_{k+1}N_1\cdots \Gamma_{k}N_i\cdots \Gamma_{k+1}N_n$ for all $i$. 
\medskip

Let us prove that (iii) implies (i). 
It is shown in \cite{AMY} that welded Milnor invariants are invariant under welded concordance. 
So it suffices to show that, if $L'$ is obtained from $L$ by surgery along a self $w_k$-tree $T$, we have $\lambda^k_i(L')\equiv \lambda^k_i(L)\mod J_i^k$ for all $i$. 
Suppose that all endpoints of $T$ are on the $i_0$th component of $L$ for some $i_0$. 
Then  $w(T)\in \Gamma_kN_{i_0}$, and by Remark \ref{rem:conj1} we have
directly that $\lambda^k_i(L)\equiv\lambda^k_i(L)\mod
\Gamma_{k+1}N_{i_0}(\subset J_i^k)$ 
for all $i\neq i_0$.
Furthermore, there exists some words $a,b$ in $F_n$,  such that
$\lambda^k_{i_0}(L)=ab$ and, again by Remark \ref{rem:conj1},
$\lambda^k_{i_0}(L')\equiv a w(T) b \mod
\Gamma_{k+1}N_{i_0}$. This  implies that $\lambda^k_{i_0}(L')\equiv \lambda^k_{i_0}(L)\mod J_{i_0}^k$, as desired. 
\medskip

Finally, we prove that (i) implies (ii). 
By assumption, and using Lemmas \ref{lem:izo} and \ref{lem:2.2}, 
the $k$--reduced $i$th preferred longitudes $\lambda^k_i(L')$ and $\lambda^k_i(L)$ differ by a finite sequence of the following operations: 
\begin{itemize}
 \item[(a)] inserting or deleting copies of $\alpha_j \ov{\alpha}_j$ or $\ov{\alpha}_j \alpha_j$ for any $j$;
 \item[(b)] inserting or deleting 
   length $\geq k+1$ commutators with entries in $\big\{\alpha_l;\ov{\alpha}_l\big\}_l$, and with at least $k+1$ entries with  some index $j\neq i$.
 \item[(c)] inserting or deleting 
   length $\geq k$ commutators with entries in $\big\{\alpha_l;\ov{\alpha}_l\big\}_l$, and with at least $k$ entries with index $i$.
\end{itemize}
We aim at showing that, under this assumption, $L$ and $L'$ are $w^{(k+1)}$-equivalent. 

As a first step, let us assume that both $L$ and $L'$ are given by a so-called sorted $w$-tree presentation:
\begin{defi}\label{def:ascending}
A $w$-tree presentation for an $n$-component welded string link is \emph{sorted}\footnote{This notion was first defined in \cite{ABMW} under the term \emph{ascending}, in the context of Gauss diagrams. 
} if, when running along each component following the orientation, all tails are met before all heads.  
\end{defi}

Given a sorted presentation $(\mathbf{1},P)$, the tails of all
$w_k$-trees are contained in the initial arcs of $\mathbf{1}$, hence
the words associated to these $w_k$-trees are 
length $k$ commutator with entries in $\big\{\alpha_l;\ov{\alpha}_l\big\}_l$.  
Denote by $P_i$ the subset of $w$-trees in $P$ whose heads are on the $i$th component. 
According to Section \ref{sec:weldedgroup}, there exists some $s_i\in \mathbb{Z}$ such that the $i$th preferred longitude is equal to $\alpha_i^{s_i}\prod_{T\in P_i} w(T)$. 
Using the Isolated move, we can introduce a union $A_i$ of $w$-arrows whose endpoints are all on the $i$th component, such that $(\mathbf{1},P\cup A_i)$ is sorted 
and equivalent to $(\mathbf{1},P)$, and such that the $i$th preferred longitude is equal to $\prod_{T\in A_i\cup P_i} w(T)$. 
By this observation, we can freely assume that the words associated with our sorted presentations, are precisely the preferred longitudes of our string links $L$ and $L'$. 
Hence by Lemma \ref{lem:word}, if $\lambda_i(L)$ and $ \lambda_i(L')$ are equal as words in $\alpha_1,\cdots,\alpha_n$, then applying (E) recursively to these sorted presentations, yields the \emph{same} union of $w$-arrows. 

It thus remains to realize the above three operations (a), (b) and (c) by a $w^{(k+1)}$-equivalence among sorted tree presentations. This is done as follows: 
\begin{itemize}
 \item Operation (a) is simply achieved by Inverse moves, which 
 insert a pair of parallel $w$-arrows, one having a $\bullet$, running from the bottom of the $j$th component and ending on the $i$th component. 
 \item
Operation (b), resp. (c), is achieved by inserting or deleting sorted  
$w$-trees, with head on the $i$th component, and with at least $k+1$ tails on the $j$th component, resp. with at least $k$ tails on the $i$th component. 
\end{itemize}
This shows that $L$ and $L'$ are $w^{(k+1)}$-equivalent, and we are done in the sorted case. 

In order to conclude the proof, it is now enough to show that any welded string link is $w^{(k+1)}$-concordant to a string link with a sorted presentation, since Milnor invariants $\mu(I)$ with $r(I)\le k$ are invariants of $w^{(k+1)}$-concordance. 
This can be seen as a direct corollary of \cite[Cor.~2.5]{Colombari}, which tells us that any $w$-tree presentation can be deformed into a sorted one by a sequence of welded concordance and surgeries along $w_{nk+1}$-trees. The desired result then follows by Remark \ref{rem:repeat}. 

\subsection{Artin-like isomorphisms}\label{sec:action}

The stacking product endows the set $\wSL(n)$ of $n$-component welded string links with a structure of monoid, with unit the trivial diagram $\1$. 
As seen in subsection \ref{sec:kred}, a conjugating 
automorphism $\varphi^{(k)}_L\in \Aut_c(\nR_kF_n)$ is associated to any welded string link $L$, 
which sends each generator $\alpha_i$ of $R_kF_n$ to its conjugate by the $k$-reduced longitude $\lambda^k_i(L)$. 
This yields, for each $k\ge 1$, a monoid homomorphism 
\[
  \varphi^{(k)}: \wSL(n)\longrightarrow \Aut_c(\nR_k F_n).
  \]
\begin{prop}\label{prop:iso}
For each $k\ge 1$, the map $\varphi^{(k)}$ descends to a group isomorphism 
\[
  \fract{\wSL(n)}{\textrm{self
    $w_k$-concordance}}\stackrel{\simeq}{\longrightarrow} \Aut_c(\nR_k F_n).
\]
\end{prop}
\begin{proof}
 Consider the quotient map $\fract{\wSL(n)}{\textrm{self $w_k$-concordance}}\longrightarrow \Aut_c(\nR_k F_n)$, which will shall still denote by $\varphi^{(k)}$. 
 The fact that this map is injective is a consequence of Theorems \ref{th:23} and \ref{thm:12}. 
 To prove surjectivity, it is sufficient to observe that an
   element of $\Aut_c(\nR_k F_n)$ is specified by an $n$-tuple of
   conjugating elements in $\nR_kF_n$, and that
   these conjugating elements are easily realized as the preferred
   longitudes of a sorted welded string link; see the paragraph
   following Def. \ref{def:ascending}. 
   Since $\Aut_c(\nR_k F_n)$ is a group, it follows that $\fract{\wSL(n)}{\textrm{self
    $w_k$-concordance}}$ is a group too\footnote{This fact was already
known, as any welded string link is invertible up to concordance
\cite[Prop. 6]{Gaudreau}.}, and that the isomorphism is a
group isomorphism.
\end{proof}

\begin{rem}
 Using the result of Colombari \cite{Colombari} mentioned in the
 introduction, the proof of Proposition \ref{prop:iso} can be adapted
 in a straightfoward way to show that, for all $k\ge 1$, we have an isomorphism 
 \[
   \fract{\wSL(n)}{\textrm{$w_k$-concordance}}\stackrel{\simeq}{\longrightarrow}
   \Aut_c\!\left(\fract{F_n}{\Gamma_{k+1} F_n}\right),
   \]
  where $\Aut_c\!\left(\fract{F_n}{\Gamma_{k+1} F_n}\right)$ denotes the group of conjugating automorphisms of $\fract{F_n}{\Gamma_{k+1} F_n}$. 
\end{rem}

We conclude this section by a (long) remark addressing the 4-dimensional counterpart of this study, that can be safely skipped by the reader who is not interested  in this matter.
\begin{rem}\label{rem:4D}
 Welded theory is intimately connected to the topology of \emph{ribbon knotted surfaces} in $4$--space, via the so-called \emph{Tube map} defined by Satoh in \cite{Satoh}. In our context, the Tube map is a surjective monoid homomorphism
 \[
   \textrm{Tube}: \wSL(n)\longrightarrow rT(n),
   \]
 where $rT(n)$ denotes the monoid of \emph{ribbon tubes} introduced in \cite{ABMW}. The question of the injectivity of this Tube map is however still
 open, but it is worth noting that our present work implies that any element in the kernel of the Tube map is self $w_k$-concordant to $\1$ for all $k$.
 
Indeed, there is a $4$--dimensional analogue of arrow calculus, which was developped (prior to \cite{arrow}) by Watanabe in \cite{Watanabe}. 
 There, a notion of $RC_k$-equivalence is defined for ribbon knotted surfaces, in terms of surgery along degree $k$ oriented claspers, which are embedded surfaces that realize topologically an oriented uni-trivalent tree with $2k$ vertices. A refined notion of \emph{self  $RC_k$-equivalence} is easily defined on $rT(n)$ by further requesting that such degree $k$ oriented claspers only intersect a single ribbon tube component. 
Combining furthermore this self $RC_k$-equivalence with the
topological concordance of ribbon tubes, one defines a notion of
\emph{self  $RC_k$-concordance} for these objects. 
It is  known \cite[Prop. 4.8]{BC} that the Tube map sends concordant welded
links to concordant ribbon tori, and it is easily verified that
it sends self $w_k$-concordant welded string links to  self
$RC_k$-concordant ribbon tubes. 
Hence, for each $k\ge 1$, the Tube map induces a surjective homomorphism
\[
 \textrm{Tube}_k: \fract{\wSL(n)}{\textrm{self
     $w_k$-concordance}}\longrightarrow \fract{rT(n)}{\textrm{self
     $RC_k$-concordance}}.
\]
These maps are actually isomorphisms, and the proof goes as
follows. Following \cite[Sec. 2.2.4]{ABMW}, which deals with the
  $k=1$ case, one can use Stallings theorem to associate an action
  $\varphi^{(k)}_T\in\Aut_c(\nR_kF_n)$ for any ribbon tube $T$,
  which actually conjugates, in $\nR_k\pi_1(B^4\setminus T)\cong\nR_kF_n$,
  the $i$th meridians of $T$ by its $i$th preferred longitude. By
  another use of Stallings theorem in dimension one more, this action
  is invariant under concordance, and it can be directly seen that
  $RC_k$-equivalent ribbon tubes have $k$-reduced $i$th longitudes which are
  congruent modulo $J_i^k$ (borrowing notation from Thm. \ref{thm:12}). As
  the Tube map is known to preserve the associated groups and longitudes, this
  leads to a map
  \[
    \varphi^{(k)}_r: \fract{rT(n)}{\textrm{self $RC_k$-concordance}}\longrightarrow\Aut_c(\nR_kF_n),
    \]
which satisfies $\varphi^{(k)}= \varphi^{(k)}_r\circ\textrm{Tube}_k$. It follows then from the injectivity of $\varphi^{(k)}$ that
$\textrm{Tube}_k$ is injective, hence an isomorphism. This implies that an element in
the kernel of the Tube map is trivial up to self $w_k$-concordance.
 \end{rem}

\section{The link case}\label{sec:links}

By Theorems \ref{thm:12} and \ref{th:23}, welded string links are classified up to self $w_k$--equivalence by their $k$--reduced longitudes. 
Similar phenomena occur in the classifications up to self-virtualization \cite{ABMW} and $w_k$--concordance \cite{Colombari}, and these results were extended to the case of welded links in \cite{AM} and \cite{Colombari} respectively, in terms of (adaptations of) the peripheral system. 
In this final section, we outline how a similar extension can be derived for links up to self $w_k$--concordance. 
\medskip

\emph{Welded links} are defined in the very same way as welded string links, by replacing the copies of the unit interval with copies of the circle $S^1$. 

Let $L$ be a welded link. 
Using the same Wirtinger-like procedure as in Section \ref{sec:milnor}, a group $G(L)$ can be associated to any welded link $L$,  which agrees with the fundamental group of the complement if $L$ is a classical link. 
A choice of one generic basepoint on each component determines a set
of meridians which normally generates $G(L)$. Up to Detour moves, these basepoints also allow to cut open $L$ into a welded string links $S_L$, and $G(L)$ is isomorphic to the quotient of $G(S_L)$ which identifies, for each strand, the meridians associated with its two endpoints. An \emph{$i$th longitude} for $L$ can then be defined
as the image of the $i$th longitude for $S_L$ in this
quotient. Of course, the result depends on the choice of the basepoints: moving the $i$th basepoint actually results in conjugating simultaneously
the associated $i$th meridian and longitude.\footnote{Note that, since all the meridians defined in this way are conjugated, all these normally generating sets of meridians lead to the same notion of $k$--reduced quotient for $G(L)$.}
 
 \begin{defi}\label{def:periph}
   Let $k$ be a positive integer. 
   The \emph{$k$--reduced peripheral system} of $L$ is defined as 
   \[
     \big(\nR_kG(L),\{x_i\}_i,\{\lambda^k_i\Gamma_kN_i\}_i\big)
   \]
where, for a given choice of basepoints, $(x_i,\lambda^k_i)$ are the images in $\nR_kG(L)$ of the associated meridians and preferred $i$th longitudes, and $\lambda^k_i.\Gamma_k N_i$ is the image in $\nR_kG(L)$ of the  coset of $\lambda^k_i$ modulo $\Gamma_k N_i$, with $N_i$ the normal subgroup of $\nR_kG(L)$ generated by $x_i$.  
It is well-defined up to, for each $i$, simultaneous conjugation of $x_i$ and $\lambda^k_i$ by an element of $\nR_kG(L)$.
 \end{defi}
\begin{rem}
The $1$--reduced peripheral system coincides with the \emph{reduced peripheral system} introduced by Milnor in his foundational paper \cite{Milnor2} on links up to link-homotopy, see also \cite{AM}. 
\end{rem}

The arrow calculus reviewed in Section \ref{sec:arrow} applies in the exact same way to welded links, and leads to the notions of \emph{self $w_k$-equivalence} and \emph{self $w_k$-concordance} for these objects, as in  Definitions \ref{def:selfwk} and \ref{def:selfwkconc}. Theorem \ref{th:nonobvious} also  holds for welded links, since the proof is purely local. 

The main result of this section is the following. 
\begin{theo} \label{thm:links}
  Two (welded) links have equivalent $k$--reduced peripheral systems if and only if they are self $w_k$--concordant. 
\end{theo}

The rest of this section is devoted to the proof of Theorem \ref{thm:links}. 
We stress that all ingredients of the proof are mostly corollaries of their string link counterparts, and straightforward adaptations of techniques of \cite{AM,Colombari}. Hence we will only sketch the proof of Theorem \ref{thm:links} below, only hinting to the main arguments and outlining the specificities of the link case.
\medskip

The fact that the $k$--reduced peripheral system is invariant under self $w_k$--concordance is an easy consequence of the string link case addressed in the previous sections. 
Indeed, a notion of $k$--reduced peripheral system can be defined for
welded string links as in Definition \ref{def:periph}, except that, in this case,
there is a natural choice of meridians, and it follows from Theorems \ref{th:23} and \ref{thm:12} that this is invariant under self $w_k$--concordance. By fixing a set of basepoints on a welded link $L$, and considering its $k$--reduced peripheral system as a quotient of the $k$--reduced peripheral system of the associated welded string link $S_L$, we obtain the desired invariance property.

The proof of \quote{only if} part of Theorem \ref{thm:links} goes roughly along the same lines as the proof of (i)$\Rightarrow$(ii) of Theorem \ref{th:23}, given in Section \ref{sec:proof13}. 
The analogue of Definition \ref{def:ascending} in the link case is the following, see \cite{AM}:  
a $w$--tree presentation for a welded link is \emph{sorted} if, on
each component, all the heads are adjacent. (A $w$--tree presentation for a welded link is a union of $w$-trees for the trivial diagram of the unlink.)

Since the $w_{kn+1}$--equivalence implies the self $w_k$--equivalence by Remark \ref{rem:repeat} and Theorem \ref{th:nonobvious}, we have the following as a direct corollary of \cite[Prop. 3.5]{Colombari}.
\begin{prop}\label{prop:4.4}
  Any welded link is self $w_k$--equivalent to a welded link admitting a sorted presentation.
\end{prop}
From this proposition, and the invariance of the $k$--reduced group under self $w_k$--equivalence, the exact same argument as
in \cite[Lem. ~1.18]{AM} proves the following. 
\begin{prop}\label{prop:CMpresentation}
  For every welded link $L$, its $k$--reduced group admits
  the following presentation
  \[
\nR_kG(L)\cong\big\langle x_1,\ldots,x_n\ |\ 
\Gamma_{k+1} N_i\textrm{ and } 
[x_i,\lambda^k_i]\textrm{ for all $i$} \big\rangle,
\]
where, for each $i$, $x_i$ is a meridian on the $i$th component, $N_i$ is the normal
subgroup generated by $x_i$, and $\lambda^k_i$ is a representative word
for the corresponding longitude.
\end{prop}

The rest of the proof  then follows the exact same lines as \cite[Prop.~3.7]{Colombari}. 
We start with two welded links which have equivalent $k$--reduced peripheral systems, and consider sorted presentations using Proposition \ref{prop:4.4}.
The main difference with the string link case of Section \ref{sec:proof13} is that a fourth operation is involved on the $i$th preferred longitudes: by Proposition \ref{prop:CMpresentation}, these longitudes might differ by the insertion or deletion of a commutator $[x_j,\lambda^k_j]$ for some $j$. 
This extra operation can be achieved by a $w_k$-concordance on sorted  presentations, 
with the same trick as illustrated in the first figure of 
\cite[Proof~of~Prop.~3.7]{Colombari}. 
This concludes the proof of Theorem \ref{thm:links}. 
\medskip 

Notice that, in the above argument, the welded concordance is only
needed for the  fourth operation that inserts/deletes a commutator
$[x_j,\lambda^k_j]$ for some $j$. This is because such relators appear in the presentation of $\nR_kG(L)$ given in Proposition \ref{prop:4.4}. 
In the special case of a link $L$ with vanishing Milnor invariants  $\ov{\mu}_L(I)$ with $r(I)\le k$, we have that  $\nR_kG(L)\cong RF_n$, meaning that this extra operation involving welded concordance is not needed for the proof. 
This implies 
that a (welded) link has vanishing Milnor invariants $\ov{\mu}(I)$ with $r(I)\le k$, if and only if it is self $w_k$-equivalent to the unlink, as stated in the introduction. 
This is shown by applying verbatim the same argument as in \cite[\S
3.2]{Colombari}.\footnote{More precisely, the exact same arguments
  showing that \cite[Thm.~3.8]{Colombari} implies
  \cite[Cor.~3.11]{Colombari}, shows that Theorem \ref{thm:links}
  implies the theorem stated at the end of the introduction.}

\bibliographystyle{abbrv}
\bibliography{References}

\end{document}